\documentclass[12pt,a4paper,reqno]{amsart}
\usepackage{amsmath,amssymb}
\usepackage{amsthm}
\usepackage{enumitem}
\usepackage[toc]{appendix}
\usepackage{multicol}
\usepackage{comment}
\usepackage{hyperref}
\usepackage{url}
\usepackage{graphicx}
\usepackage{mathtools}
\usepackage{color}

\def\veca{{\mbox{\boldmath $a$}}}
\def\vecb{{\mbox{\boldmath $b$}}}
\def\vecc{{\mbox{\boldmath $c$}}}

\def\vecm{{\mbox{\boldmath $m$}}}

\def\vec0{\mbox{\bf 0}}

\def\tr{\mathop{\rm tr }\nolimits}

\newcommand{\abs}[1]{|#1|}

\def\sp{\mathop{\rm sp }\nolimits}

\DeclareMathOperator{\rank}{rank}

\newcommand{\RR}{\mathds{R}}
\newcommand{\CC}{\mathds{C}}

\newcommand{\NN}{\mathds{N}}

\def\P{{\mbox {\boldmath $P$}}}

\def\matrix0{{\mbox {\boldmath $O$}}}

\def\G{$G$}

\usepackage{listings}
\usepackage{tikz}
\tikzset{
every node/.style={circle, inner sep=2pt}
}
\usetikzlibrary{matrix,arrows,calc}

\newtheorem{theorem}{Theorem}
\newtheorem{lemma}[theorem]{Lemma}
\newtheorem{proposition}[theorem]{Proposition}
\newtheorem{corollary}[theorem]{Corollary}

\theoremstyle{definition}
\newtheorem{definition}[theorem]{Definition}
\newtheorem{observation}[theorem]{Observation}
\newtheorem{remark}[theorem]{Remark}
\newtheorem{example}[theorem]{Example}
\newtheorem{problem}{Problem}
\newtheorem{question}[theorem]{Question}

\def \rank {\operatorname{rank}}

\def \RR {\mathbb{R}}
\def \NN {\mathbb{N}}
\def \CC {\mathbb{C}}
\def \G {\mathcal{G}}

\def \P {\mathcal{P}}

\def \mC {\mathcal{C}}

\def\1{{\bf 1}}

\usepackage{xcolor}

\title[]{Eigenvalue bounds for the quantum chromatic number of graph powers}
\author{Aida Abiad}\thanks{\texttt{a.abiad.monge@tue.nl},  Department of Mathematics and Computer Science, Eindhoven University of Technology, The Netherlands,  Department of Mathematics and Data Science, Vrije Universiteit Brussel, Belgium}
\author{
Benjamin Jany}\thanks{\texttt{b.jany@tue.nl},  Department of Mathematics and Computer Science, Eindhoven University of Technology, The Netherlands}

\begin{document}
\maketitle

\begin{abstract}
The quantum chromatic number, a generalization of the chromatic number, was first defined in relation to the non-local quantum coloring game. We generalize the former by defining the quantum $k$-distance chromatic number $\chi_{kq}(G)$ of a graph $G$, which can be seen as the quantum chromatic number of the $k$-th power graph, $G^k$, and as generalization of the classical $k$-distance chromatic number $\chi_k(G)$ of a graph. It can easily be shown that $\chi_{kq}(G) \leq \chi_k(G)$. 
In this paper, we strengthen three classical eigenvalue bounds for the $k$-distance chromatic number by showing they also hold for the quantum counterpart of this parameter. This shows that several bounds by Elphick et al. [J. Combinatorial Theory Ser. A 168, 2019, Electron. J. Comb. 27(4), 2020] hold in the more general setting of distance-$k$ colorings. As a consequence we obtain several graph classes for which $\chi_{kq}(G)=\chi_{k}(G)$, thus increasing the number of graphs for which the quantum parameter is known. \\

\end{abstract}

%%%%%%%%%%%%%%%%%%%%%%%%%%%%%%%%%%%%%%%%%%%%%%%%%%%%%%%%%%%
\section{Introduction}
%%%%%%%%%%%%%%%%%%%%%%%%%%%%%%%%%%%%%%%%%%%%%%%%%%%%%%%%%%%

Quantum graph parameters originated in the context of non-local games. The latter are game-like models involving two or more cooperative players, whose objective is to win the game without being able to communicate with each other throughout the duration of the game.
The players are only allowed to build a common strategy before the game starts.
It turns out that by sharing an entangled quantum state, players can increase their probability of winning such non-local games. In this paper, we focus on the quantum chromatic number, a graph  parameter related to a specific non-local game: the graph coloring game.

The graph colouring game is as follows. Given a graph $G$, two players, Alice and Bob, are each given a vertex of the graph, and each must respond with a integer in $[c] := \{1, \ldots, c\}$. To win the game, the players must answer differently if their given vertices were adjacent, or answer identically if their vertices were equal. In a classical setting (i.e. no quantum states involved), Alice and Bob can win with certainty if $c \geq \chi(G)$, where $\chi(G)$ is the chromatic number of the graph. However, if the players were to share an entangled quantum state, they 
can, depending on the graph chosen, 
win the game with certainty even if $c < \chi(G)$. The quantum chromatic number  is precisely the smallest integer $c$, for which Alice and Bob can win the graph coloring game with certainty, when allowed to share an entangled quantum state. This parameter  appeared first in \cite{A2005} 
%in the context of the graph colouring game and 
building on \cite{CHTW2004,GW2002}.  %The former was defined as the smallest integer, for two players to win with certainty the graph colouring game if the two players share an entangled quantum state
It has since received quite some attention, see e.g. \cite{elphick19,GS2024, L2024,MR2016,MRo2016,elphick19part2,Paulsen16}.

For a positive integer $k$, the $k^{\text{th}}$ \emph{power of a graph} $G =(V, E)$, denoted by $G^k$, is a graph with vertex set $V$ in which two distinct elements of $V$ are joined by an edge if there is a path in $G$ of length at most $k$ between them. Problems related to the chromatic number $\chi(G^k)$
of power graphs $G^k$ were first considered by Kramer and Kramer in \cite{KK69,KK69v2} in 1969 and have enjoyed significant attention ever since then.

%The $k$-\emph{distance chromatic number}, denoted by $\chi_k(G)$, is just the chromatic number of $G^k$. Hence, $\chi_k(G)=\chi_1(G^k)$. However, even the simplest algebraic or combinatorial parameters of the power graph $G^k$ cannot be deduced easily from the corresponding parameters of the graph $G$. For instance, neither the spectrum~\cite{Das2013LaplacianGraph},~\cite[Section 2]{acfns20}, nor the average degree~\cite{Devos2013AveragePowers}, nor the rainbow connection number~\cite{Basavaraju2014RainbowProducts} of $G^k$ can be derived in general directly from those of the original graph $G$. 

We investigate a quantum analogue of the distance chromatic number $\chi_{k}(G)$. A natural extension of the combinatorial definition of the quantum chromatic number (see \cite[Definition 1]{MRo2016}) to the quantum $k$-distance chromatic number is as follows.

 \begin{definition}\label{def:quantumdistancechromaticnumber}
A \emph{quantum $k$-distance $c$-coloring} of a graph $G = (V,E)$, denoted by $\chi_{kq}$, is a collection of orthogonal
projectors $\{P_{v,h} : v \in V, h \in [c]\}$ in $\mathbb{C}^{d\times d}$ such that
\begin{itemize}
    \item for all vertices $v\in V$
    \begin{equation}\label{deffirstcondition}
        \sum_{h\in [c]} P_{v,h} = I_d \quad \text{(completeness),}
    \end{equation}
        \item for all vertices $v,w\in V$ with $\text{dist}(v,w)\leq k$ and for all $h\in [c]$
    \begin{equation}\label{defsecondcondition}
         P_{v,h}P_{w,h} = 0_d \quad \text{(orthogonality).}
    \end{equation}
\end{itemize}
The \emph{quantum $k$-distance chromatic number} $\chi_{kq}(G)$ is the smallest $c$ for which the graph $G$ admits a quantum $k$-distance $c$-coloring for some dimension $d > 0$.
\end{definition}

Note that for $k=1$ we obtain the definition of quantum $c$-coloring from \cite[Definition 1]{MRo2016}. Observe also that any classical $c$-coloring can be viewed as a $1$-dimensional quantum coloring (i.e.~letting $d= 1$), where we set $P_{v,h} = 1$ if vertex $v$ has color $h$ and we set $P_{v,h} = 0$, otherwise. Thus, a quantum coloring is a relaxation of the classical coloring. Also, observe that
\begin{equation}\label{rmk:kqcoloringandpowergraph}
\chi_{kq}(G) = \chi_q(G^{k}).
\end{equation}

Despite the above, even the simplest algebraic or combinatorial parameters (including the eigenvalues) of the power graph $G^k$ cannot be easily deduced from the corresponding parameters of the graph $G$, see e.g.\cite{acfns20,Basavaraju2014RainbowProducts,Das2013LaplacianGraph,Devos2013AveragePowers}. In this regard, several eigenvalue bounds on $\chi_k(G)$ that depend only on the spectrum of $G$ have been proposed in the literature. Notably, two inertial-type bounds were shown by Abiad et al.~in \cite{acfns20} (see Theorems \ref{thm:1st_inertial_chi} and \ref{theo:extended-EW}). 
The same authors also prove a Hoffman ratio-type bound on $\chi_k(G)$ (see Theorem \ref{thm:ratio_chi}).
These three eigenvalue bounds are shown to be sharp for several graph classes. The quality of these bounds depends on the choice of a degree-$k$ polynomial, so finding the best possible lower bound for a given graph is, in fact, an optimization problem which is investigated in \cite{acfns20,ANR2025}.

In this work, we show that the three eigenvalue bounds on the classical parameter $\chi_k(G)$ also hold in the quantum setting. It is not known whether the quantum counterpart of $\chi_k(G)$ is a computable function. As a consequence of our results, we can use the bounds optimization shown in the classical setting to obtain several graph classes for which $\chi_{k}(G) =\chi_{kq}(G)$, thus increasing the number of graphs for which the quantum parameter is known. Our work extends several known results from \cite{acfns20,elphick17,elphick19,elphick19part2}. While an application of the Hoffman ratio-type bound on $\chi_k(G)$ to coding theory has been recently presented \cite{ANR2025}, our results show the first quantum application of the three aforementioned eigenvalue bounds.

 \medskip

%%%%%%%%%%%%%%%%%%%%%%%%%%%%%%%%%%%%%%%%%%%%%%%%%%%%%%%%%%%
\section{Three eigenvalue bounds in the classical setting}\label{sec:literature}
%%%%%%%%%%%%%%%%%%%%%%%%%%%%%%%%%%%%%%%%%%%%%%%%%%%%%%%%%%%

 In this section, we recall several bounds on the classical $k$-distance chromatic number of a graph which use the eigenvalues of the adjacency matrix. In Section \ref{sec:inirtia} and \ref{sec:hoffman}, we will show that all these bounds hold for the quantum $k$-distance chromatic number as well.

Recall that the \emph{inertia} of a graph $G$ is the ordered triple $(n^{+}, n^0, n^{-})$ where $n^{+}$, $n^{0}$ and $n^{-}$ are the numbers of positive, zero and negative eigenvalues of the adjacency matrix $A$, including multiplicities.

The first bound is derived by Abiad, Coutinho and Fiol, from the Inertial-type bound in \cite[Theorem 3.1]{acf2019} using the fact that upper bounds on $\alpha_k(G)$ directly yield lower bounds on $\chi_k(G)$ using the fact that $\chi_k(G)\geq \frac{n}{\alpha_k(G)}$.

\begin{theorem}[First inertial-type bound]\cite[Theorem 3.1]{acf2019}\label{thm:1st_inertial_chi}
        Let $G$ be a graph with adjacency matrix $A$ having eigenvalues $\lambda_1 \geq \cdots \geq \lambda_n$. Let $p \in \mathbb{R}_k[x]$ with corresponding parameters $W(p) := \max_{u \in V} \{(p(A))_{uu}\}$ and $w(p) := \min_{u \in V} \{(p(A))_{uu}\}$. Then,  
        \begin{equation}\label{firstinertialtypechikAbiadetal}
            \chi_k(G) \geq \frac{n}{\min\{ \abs{\{i : p(\lambda_i) \geq w(p) \}},\abs{\{i : p(\lambda_i) \leq W(p) \}}\}}.
        \end{equation}
    \end{theorem}

 Similarly, one can use the Ratio-type bound for $\alpha_k(G)$ which appears in \cite[Theorem 3.2]{acf2019} and obtain:
 
    \begin{theorem}[Hoffman ratio-type bound]\cite[Theorem 4.3]{acfns20}\label{thm:ratio_chi}
        Let $G$ be a graph with $n$ vertices and adjacency matrix $A$ having eigenvalues $\lambda_1  \geq \cdots \geq \lambda_n$. Let $p \in \mathbb{R}_k[x]$ with corresponding parameters $W(p) := \max_{u \in V}$ $\{(p(A))_{uu}\}$ and $ \lambda(p):= \min_{i\in[2,n]} \{p(\lambda_i)\}$, and assume $p(\lambda_1) > \lambda(p)$. Then, 
        \begin{equation}\label{eq:ratio_chi}
            \chi_k(G) \geq \frac{p(\lambda_1) - \lambda(p)}{W(p) - \lambda(p)}.
        \end{equation}
    \end{theorem}
    
For $k=1$ the above gives the celebrate Hoffman bound on the chromatic number, $\chi(G)\geq 1-\frac{\lambda_1}{\lambda_n}$.

In \cite{acfns20}, yet another stronger inertial-type bound for $\chi_k$ is shown by assuming $k$-partially walk-regularity:

\begin{theorem}[Second inertial-type bound]\cite[Theorem 4.2]{acfns20}	\label{theo:extended-EW}
Let $G$ be a $k$-partially walk-regular graph with adjacency  eigenvalues $\lambda_1\geq \cdots \geq \lambda_n$. Let $p\in \mathbb{R}_k[x]$. Then,
%such that $\sum_{i=1}^{n} p(\lambda_i) = 0$\aida{this was the redundant assumption right? if yes, let us remove}. Then,
\begin{equation}\label{extendedElphickWocjan}
	\chi_k(G)\geq 1+\max{ \left(\frac{|j:p(\lambda_j)< 0|}{|j:p(\lambda_j)> 0|}\right)}.
\end{equation}
\end{theorem}

\begin{remark}
    Note in the original statement of the above theorem, the additional constraint $\sum_{i=1}^{n} p(\lambda_i) = 0$ was present. However, we manage to show in the proof of Theorem \ref{thmquantum:2nd_inertial_chi} that the constraint is redundant, and therefore omit it from the statement. 
\end{remark}

Theorem \ref{theo:extended-EW} is an extension of \cite[Theorem 1]{elphick19} (note for the $k=1$ case walk-regularity is no needed).

\begin{theorem}\label{theo:bound-EW}\cite[Theorem 1]{elphick19}
Let $G$ be a graph with inertia $(n^{+}, n^0, n^{-})$. Then,
\begin{equation}
\label{bound-EW}
\chi(G) \ge 1 + \max \left(\frac{n^+}{n^-} , \frac{n^-}{n^+}\right).
\end{equation}
\end{theorem}
%Note that \eqref{bound-EW} holds with equality for the two bounds only if $n^0 = 0$, since
%\[
%1 + \max \left(\frac{n^+}{n^-} , \frac{n^-}{n^+}\right) =
%\frac{n^+  + n^-}{\min\{n^+ ,  n^-\}}.
%%\frac{|i: \lambda_i > 0| + |i : \lambda_i <0|}{\min\{|i : \lambda_i > 0| , |i : \lambda_i < 0|\}}.
%\]

In \cite{elphick19}, Elphick and Wocjan, also show that the inertial-type bound \eqref{bound-EW} for $\chi$ is in fact also a lower bound for the corresponding quantum chromatic parameter $\chi_q$. %Moreover, some stronger lower bounds for $\chi_q$ are shown in \cite{elphick19part2}. 

 \medskip

%%%%%%%%%%%%%%%%%%%%%%%%%%%%%%%%%%%%%%%%%%%%%%%%
\section{Pinching and the quantum $k$-distance coloring}\label{sec:pinching}
%%%%%%%%%%%%%%%%%%%%%%%%%%%%%%%%%%%%%%%%%%%%%%%%

In \cite{elphick19}, the authors establish the existence of a quantum coloring using the pinching operation of a suitable set of orthogonal projectors. We mimic this proof in order to determine whether a set of orthogonal projector forms a quantum $k$-distance coloring. 

First we recall the pinching operation. 

\begin{definition}[Pinching]
    Let $\P := \{P_s \in \CC^{d \times d}\, : \, s \in [c]\}$ be a collection of orthogonal projectors that form a resolution of the identity matrix. For all $X \in \CC^{d \times d}$, we define the \emph{pinching of $X$ by $\P$} to be
    $$\mC_{\P}(X) := \sum_{s \in [c]} P_sXP_s.$$
    We say the pinching $\mC_{\P}$ annihilates $X$ if $\mC_{\P}(X) = 0$. 
\end{definition}

We now focus on the desired result. 

\begin{theorem}\label{thm:coltopinch}
    Let $\{P_{v,s}: v \in V, s \in [c]\}$ be a quantum $k$-distance coloring of $G$ in $\CC^{d \times d}$. Then the following block-diagonal orthogonal projectors
    $$P_s := \sum_{v \in V} e_ve_v^{\dag} \otimes P_{v,s} \in \CC^{n\times n} \otimes \CC^{d\times d},$$
    whose collection we denote by $\P$, form a resolution of the identity matrix and the corresponding pitching operation $\mC_{\P}$:
    \begin{itemize}
        \item[1)] annihilates $A^{\ell} \otimes I_d$, for $1 \leq \ell \leq k$,
        \item[2)] satisfies $\mC_{\P}(E \otimes I_d) = E \otimes I_d$ for all diagonal matrices $E \in \CC^{n \times n}$.
    \end{itemize}
\end{theorem}

\begin{proof}
Since $\{P_{v,s} \, : \, v \in V, s \in [c]\}$ is a quantum $k$-distance coloring, it is also a quantum coloring. Therefore by \cite[Theorem 1]{elphick19}, the collection $\P$ forms a resolution of the identity and 2) holds as well. Hence we are left to show that $\mC_{\P}(A^{\ell} \otimes I_d) = 0$ for all $1 \leq \ell \leq k$. Fix such an $\ell$. 

\begin{align*}
   & \mC_{\P}(A^{\ell} \otimes I_d) \\
   &= \sum_{s \in [c]} P_s(A^{\ell} \otimes I_d)P_s \\
   &= \sum_{s \in [c]}\left( \sum_{v \in V} e_ve_v^{\dag} \otimes P_{v,s} \right)(A^{\ell} \otimes I_d) \left( \sum_{w \in V} e_we_w^{\dag} \otimes P_{v,s} \right) \\
   &= \sum_{s \in [c]} \sum_{v,w \in V} A^{\ell}_{vw} \cdot e_{v}e_{w}^{\dag} \otimes P_{v,s}P_{w,s}\\
   &= \sum_{s \in [c]} \left(\sum_{\substack{v,w \in V \\ d(v,w) \leq k}} A^{\ell}_{vw} \cdot e_{v}e_{w}^{\dag} \otimes 0_d + \sum_{\substack{v,w \in V \\ d(v,w) > k}} 0 \cdot e_{v}e_{w}^{\dag} \otimes P_{v,s}P_{w,s} \right) \\
   &= 0.
\end{align*}
The second to last equality follows from the fact that $A^{\ell}_{v,w}$ counts the number of walks of length $\ell$ between $v$ and $w$, hence if $d(v,w) > k$ then $A^{\ell}_{v,w} = 0$ for all $\ell \leq k$. If $d(v,w) \leq \ell$ then $P_{v,s}P_{w,s} = 0$ for all $s \in [c]$ by definition of quantum $k$-distance coloring. 
\end{proof}

We now establish a converse statement similar to that of \cite[Theorem 2]{elphick19}.

\begin{theorem}
    Assume there exists a pinching $\mC_{\P}$ where $\P := \{ P_s \in \CC^{nd \times nd} \, : \, 1 \leq s \leq c\}$ such that $\P$ forms a resolution of the identity, $\mC_{\P}(A^{\ell} \otimes I_d) = 0$ for all $1 \leq \ell \leq k$ and $\mC_{\P}(E \otimes I_d) = E \otimes I_d$ for all diagonal matrices $E \in \CC^{n \times n}$. Then there exists a quantum $k$-distance coloring of $G$. 
\end{theorem}

\begin{proof}
    Using the same argument as the proof of \cite[Theorem 2]{elphick19}, we must have that $P_s$ are block diagonal for all $s \in [c]$, and the block are indexed by the vertices $v$ of $G$.  We refer to each of those block as $P_{v,s}$ for $v \in V$ and $s \in [c]$. Because $P_s$ is an orthogonal projector, it can easily be seen that each $P_{s,v}$ for $v \in V$ are also orthogonal projectors. Moreover, since $\P$ forms a resolution of the identity then for each $v \in V$ we must have $\sum_{s \in [c]} P_{v,s} = I_d$.

    Because $\mC_{\P}(A^{\ell} \otimes I_d) = 0$ for all $\ell \leq k$ then following the operations of pinching we get

$$\sum_{\substack{v,w \in V \\ d(v,w) \leq k}} A_{vw}^{\ell} \cdot e_v e_w^{\dag} \otimes \left( \sum_{s \in [c]} P_{v,s}P_{w,s} \right) = 0.$$

However, if $d(v,w) \leq k$ then $A^{\ell}_{v,w} > 0$. Hence it must be that 
    $\sum_{s \in [c]} P_{v,s}P_{w,s} = 0$ for all $s \in [c]$. Finally multiplying on the left by $P_{v,t}$ and on the right by  $P_{w,t}$ for an arbitrary $t \in [c]$ we get that $P_{v,t}P_{w,t} = 0$. Thus we have an $k$-distance coloring $\{P_{v,s} \, : \, s \in [c], v \in V\}$.
\end{proof}

 \medskip

%%%%%%%%%%%%%%%%%%%%%%%%%%%%%%%%%%%%%%%%%%%%%%%%%%%%%%%%%%%
\section{Inertial-type bounds for the quantum distance-$k$ chromatic number} \label{sec:inirtia}
%%%%%%%%%%%%%%%%%%%%%%%%%%%%%%%%%%%%%%%%%%%%%%%%%%%%%%%%%%%

%%%%%%%%%%%%%%%%%%%%%%%%%%%%%%%%%%%%%%%%%%%%%%%%%%%%%%%%%%%
\subsection{First inertial-type bound}
%%%%%%%%%%%%%%%%%%%%%%%%%%%%%%%%%%%%%%%%%%%%%%%%%%%%%%%%%%%

The following bound on $\chi_{kq}$ extends Theorem \ref{thm:1st_inertial_chi} to its quantum counterpart:

\begin{theorem}[First inertial-type  bound]\label{thmquantum:1st_inertial_chi}
        Let $G$ be a graph of order $n$ with adjacency matrix $A$ having eigenvalues $\lambda_1 \geq \cdots \geq \lambda_n$. Let $p \in \mathbb{R}_k[x]$ with corresponding parameters $W(p) := \max_{u \in V} \{(p(A))_{uu}\}$ and $w(p) := \min_{u \in V} \{(p(A))_{uu}\}$. Then, 
        \begin{equation*}
            \chi_{kq}(G) \geq \frac{n}{\min\{ \abs{\{i : p(\lambda_i) \geq w(p) \}},\abs{\{i : p(\lambda_i) \leq W(p) \}}\}}.
        \end{equation*}
    \end{theorem}

\begin{proof}
    The result follows immediately using the fact that $\chi_{kq} \geq \frac{n}{\alpha_{kq}}$ and \cite[Theorem 3.3]{WEA2022}.
\end{proof}

Details on the optimization of this bound can be found in \cite{acfns20}. As a corollary of Theorem \ref{thmquantum:1st_inertial_chi}, for $k=1$ we obtain \cite[Theorem 3.1]{acf2019}.

Using the extension of the inertial-type quantum bound for $\alpha_{kq}$ from \cite[Section 5]{GGL}, one obtains:

\begin{corollary}\label{coro:firstinertialinifinitedimprojectors}
    The bound from Theorem \ref{thmquantum:1st_inertial_chi} still holds when considering infinite dimensional projectors.
\end{corollary}

The bound from Theorem \ref{thmquantum:1st_inertial_chi} holds with equality for several graph classes, see \cite[Section 2]{acfns20}.

%%%%%%%%%%%%%%%%%%%%%%%%%%%%%%%%%%%%%%%%%%%%%%%%%%%%%%%%%%%
\subsection{Second inertial-type bound}
%%%%%%%%%%%%%%%%%%%%%%%%%%%%%%%%%%%%%%%%%%%%%%%%%%%%%%%%%%%

Next we show that the inertial bound from Theorem \ref{theo:extended-EW} is also a lower bound for quantum $k$-distance chromatic number $\chi_{kq}(G)$ in the case when $G$ is $k$-partially walk-regular. This, in turn, extends the bound for $\chi_k$ from \cite[Theorem 4.2]{acfns20} to its quantum counterpart.

Before stating the following results, note that for all $p \in \RR[x]$ we have that $p(A \otimes I) = p(A) \otimes I$. This follows directly from properties of the Kronecker product.

\begin{theorem}[Second-inertial type bound]	\label{thmquantum:2nd_inertial_chi}
Let $G$ be a $k$-partially walk-regular graph with adjacency  eigenvalues $\lambda_1\geq \cdots \geq \lambda_n$. Let $p\in \mathbb{R}_k[x]$. Then,  
\begin{equation}\label{quantumextendedElphickWocjan}
	\chi_{kq}\geq 1+\max{ \left(\frac{|j:p(\lambda_j)< 0|}{|j:p(\lambda_j)> 0|}\right)}.
\end{equation}
\end{theorem}

\begin{proof}
Let $\{P_{v,s} \, : \, v \in V, s \in [c]\}$ be a quantum $k$-distance coloring of $G$ and $\P = \{P_1, \ldots, P_c\}$ be the orthogonal projection as defined in Theorem \ref{thm:coltopinch}. Let $\omega = e^{2\pi i/c}$ and define 
$$U := \sum_{s \in [c]} \omega^s P_s.$$
It is easy to show that $U$ is a unitary matrix, and hence $U^{\ell}$ is unitary for all $\ell \in \NN$. It was shown in \cite[Lemma 1]{elphick19} that for any $X \in \CC^{nd \times nd}$, 
$$\mC_{\P}(X) = \frac{1}{c}\sum_{s \in [c]} U^{\ell} X (U^{\dag})^{\ell}.$$
Hence by Theorem \ref{thm:coltopinch}, for all $p \in \RR_k[x]$ we have 
$$-p(A) \otimes I = \sum_{s =1}^{c-1} U^{\ell} (p(A) \otimes I)(U^{\dag})^{\ell}.$$

Let $v_1, \ldots, v_n$ be the eigenvectors of unit length corresponding to the eigenvalues $p(\lambda_1) \geq \cdots \geq p(\lambda_n)$ of $p(A)$. Note that $p(A) \otimes I_d$ has for eigenvalue $p(\lambda_i)$ with multiplicity $d$ and this for all $1 \leq i \leq n$.
Furthermore, $\{v_i \otimes e_j \, :, 1 \leq j \leq d\}$, for $1 \leq i \leq n$ are eigenvectors with eigenvalue $p(\lambda_i)$.
Let $p(A) \otimes I_d = p(B) - p(C)$, where
\begin{align*}
    p(B) &= \sum_{i=1}^{|\{j \, : \, p(\lambda_j) > 0\}|}\sum_{\ell=1}^d p(\lambda_i)(v_i \otimes e_{\ell})(v_i \otimes e_{\ell})^*\\ 
    p(C) &= \sum_{i=n - |\{j \, : \, p(\lambda_j) > 0\}|+1}^{n}\sum_{\ell=1}^d - p(\lambda_i)(v_i \otimes e_{\ell})(v_i \otimes e_\ell)^*
\end{align*}

Note $p(B)$ and $p(C)$ positive semidefinite matrices, and we know $\rank(p(B)) = d|\{j \, : \, p(\lambda_j) > 0\}|$ and $\rank(p(C)) = d|\{j \, 
: \, p(\lambda_j) < 0\}|$.  

Let $P^+$ and $P^-$ be the orthogonal projections onto the subspaces spanned by the eigenvectors corresponding to the positive and negative eigenvalues of $p(A) \otimes I_d$, that is:

\begin{align*}
   P^+ &=  \sum_{i=1}^{|\{j \, : \, p(\lambda_j) > 0\}|}\sum_{\ell=1}^d (v_i \otimes e_{\ell})(v_i \otimes e_{\ell})^* \\
   P^- &= \sum_{i=n - |\{j \, : \, p(\lambda_j) > 0\}|+1}^{n}\sum_{\ell=1}^d (v_i \otimes e_{\ell})(v_i \otimes e_\ell)^*.
\end{align*} 
Then 
$$p(B) = P^+(p(A) \otimes I_d) P^+ \quad \textrm{and} \quad p(C) = -P^-(p(A) \otimes I_d) P^-,$$
and therefore we get 
$$\sum_{i=1}^c U_i p(B) U_i^{\dag} -\sum_{i=1}^c U_i p(C) U_i^{\dag} = p(C) - p(B).$$

Multiplying both sides by $P^-$ we get 

$$P^-\sum_{i=1}^c U_i p(B) U_i^{\dag}P^- - P^-\sum_{i=1}^c U_i p(C) U_i^{\dag}P^- = p(C)$$

Since $P^-\sum_{i=1}^c U_i p(C) U_i^{\dag}P^-$ is positive semi-definite, we obtain 

$$P^-\sum_{i=1}^c U_i p(B) U_i^{\dag}P^- \geq p(C).$$

Playing with ranks of sums and rank of products together with \cite[Lemma 2]{elphick17}, we get that
\begin{align*}
&(c-1)d|\{j \, : \, p(\lambda_j) > 0\}| \geq d|\{j \, : \, p(\lambda_j) < 0\}|.
\qedhere\end{align*}
\end{proof}

Note the above bound can be optimized by finding a suitable polynomial $p$. Although a closed formula for optimal polynomials is unknown in most cases, in \cite{ANR2025} a linear program that optimizes \eqref{quantumextendedElphickWocjan} was proposed. Furthermore, as an immediate corollary of Theorem \ref{thmquantum:2nd_inertial_chi}, for $k = 1$, we get that  \cite[Theorem 4.2]{acfns20}. Finally the second inertia bound for the classical $k$-distance chromatic number is known to be tight for certain classes of graph such as the (generalized) Petersen graphs with $(n,k) \in \{(5, 2),(8, 3),(10, 2)\}$ (see \cite[Section 3]{acfns20}). Hence for those graphs and specified values of $k$, the quantum $k$-distance chromatic is equal to the $k$-distance chromatic number.

% \begin{remark}\label{rem:redundantassumption}
% \aida{I would add this in section 2, fixing the Theorem \ref{theo:extended-EW} there as well, as a note that we observed that that assumption is redundant, and remove the corollary below (now it is confusing).}
%     We would like to remark that the proof of Theorem \ref{thmquantum:2nd_inertial_chi} generalizes the proof of \cite[Theorem 4.2]{acfns20}. In fact by letting $\P$ be the collection of orthogonal projections corresponding to a classical $k$-coloring, we can derive \eqref{extendedElphickWocjan} without the condition that $\sum_{i=1}^n p(\lambda_i) = 0$. We therefore get the following corollary for the classical $k$-distance chromatic number which only differs from \cite[Theorem 4.2]{acfns20} by omitting the constraint on $p \in \RR_k[x]$.
% \end{remark}

% \begin{corollary}[Second inertial-type bound]
% Let $G$ be a $k$-partially walk-regular graph with adjacency  eigenvalues $\lambda_1\geq \cdots \geq \lambda_n$, and $p\in \mathbb{R}_k[x]$. Then,
% \begin{equation}
% 	\chi_k(G)\geq 1+\max{ \left(\frac{|j:p(\lambda_j)< 0|}{|j:p(\lambda_j)> 0|}\right)}.
% \end{equation}
% \end{corollary}

 \medskip

%%%%%%%%%%%%%%%%%%%%%%%%%%%%%%%%%%%%%%%%%%%%%%%%%%%%%%%%%%%
\section{Hoffman ratio-type bound for the quantum distance-$k$ chromatic number} \label{sec:hoffman}
%%%%%%%%%%%%%%%%%%%%%%%%%%%%%%%%%%%%%%%%%%%%%%%%%%%%%%%%%%%

The authors from \cite{elphick19} prove several Hoffman ratio-type bounds (see \cite[Eq. (3)]{elphick19}), like the well-known Hoffman bound on the chromatic number of a graph, $\chi(G)\geq 1-\frac{\lambda_1}{\lambda_n}$, hold for the quantum chromatic number of a graph as well. 

In this section we show that Hoffman ratio-type bounds for the $k$-distance chromatic number also hold for its quantum counterpart. 
We will do so through the intermediate of quantum homomorhpisms and the Lov\'asz theta number of a graph.
As we will only need a few properties of the latter concepts to prove our bound, we refer the reader to the work of Man\u{c}inska and Roberson \cite{mancinska162}, for precise definitions. 
We write $G \xrightarrow q H$ if there exist a quantum homomorphism from $G$ to $H$. Furthermore we denote by $\overline{\vartheta}(G)$ to be the Lov\'asz theta number of the complement of $G$.

We will need the following theorem.

\begin{theorem}\cite[Theorem 3.2]{mancinska162} \label{thm:quantumboundtheta}
If $G \xrightarrow q H$ then $\overline{\vartheta}(G) \leq \overline{\vartheta}(H)$.
\end{theorem}

The quantum chromatic number can be determined using quantum homomorphism in the following way (see \cite[Section 4]{mancinska162} for more details) : $$\chi_q(G) = \min \{n \in \NN \, : \, G \xrightarrow q K_n\},$$ 
where $K_n$ is the complete graph on $n$ vertices. Furthermore recall that $\overline{\vartheta}(K_n) = n$.

We are now ready to prove the Hoffman ratio-type bound for quantum $k$-distance chromatic number. 

\begin{theorem}[Hoffman ratio-type bound]\label{thm:ratiotype}
    Let $G$ be a graph on $n$ vertices with adjacency matrix $A$ having eigenvalues $\lambda_1 \geq \cdots \geq \lambda_n$. Let $p \in \RR_k[x]$ such that $p(\lambda_1) > p(\lambda_i)$ for all $i \in [2,n]$. Then,
    \begin{equation}\label{eq:qtmratiotype}
        \chi_{kq}(G) \geq \frac{p(\lambda_1) - \lambda(p)}{W(p) - \lambda(p)},
    \end{equation}
    where $W(p) := \max_{u \in V} \{ (p(A))_{uu}\}$ and $\lambda(p) := \min_{i \in [2,n]} \{p(\lambda_i)\}$.
\end{theorem}

\begin{proof}
    Let $c := \chi_{kq}(G)$. By \eqref{rmk:kqcoloringandpowergraph}, we also have $c = \chi_q(G^{k})$. By definition of the quantum chromatic number (see Definition \ref{def:quantumdistancechromaticnumber} for $k=1$) this implies $G^{k} \xrightarrow q K_c$ and hence, by Theorem \ref{thm:quantumboundtheta}, $\overline{\vartheta}(G^{k}) \leq \overline{\vartheta}(K_c) = c = \chi_{kq}(G)$.
    
    Finally it is known (see \cite{acf2019}) that 
    $$\frac{p(\lambda_1) - \lambda(p)}{W(p) - \lambda(p)} \leq \overline{\vartheta}(G^{k}).$$
    Combined with the above this proves the desired inequality.    
\end{proof}

Similarly to the second inertia bound, one needs to establish the polynomial $p \in \RR_K[x]$ that optimizes \eqref{eq:qtmratiotype}. For $k=2,3$, closed formulas with the optimal polynomials are derived in \cite[Section 2.2]{ANR2025}. For all other $k$, the optimal polynomials are still unknown. However, the authors of \cite{ANR2025} establish a linear program that can be used to optimize \eqref{eq:qtmratiotype}. 

%Details on the optimization of the bound from Theorem \ref{thm:ratiotype} for general $k$ can be found in  and also in the Appendix. .

Finally, as an immediate corollary we obtain for $k=1$ \cite[Theorem 4.3]{acfns20}.

 \medskip

%%%%%%%%%%%%%%%%%%%%%%%%%%%%%%%%%%%%%%%%%%%%%%%%%%%%%%%%%%%
\section{Concluding remarks}
%%%%%%%%%%%%%%%%%%%%%%%%%%%%%%%%%%%%%%%%%%%%%%%%%%%%%%%%%%%

We showed that three classical eigenvalue lower bounds for the \linebreak distance-$k$ chromatic number are also lower bounds for the corresponding quantum parameter. The quality of such bounds depends on the choice of a polynomial, so finding the best possible upper bound for a given graph is in fact an optimization problem. Since such optimization been studied for the classical case (see \cite{acfns20,ANR2025}), as a consequence of our results, now one can use the existing optimization methods to compute the best eigenvalue bound for the quantum distance chromatic number. This allows us to obtain several graph classes for which $\chi_{k}(G) =\chi_{kq}(G)$, thus increasing the number of graphs for which the quantum parameter is known. 

Indeed, since we know that $\chi_{kq}\leq \chi_k$, then an immediate consequence of the new Theorems \ref{thm:ratio_chi}, \ref{thmquantum:2nd_inertial_chi} and \ref{thm:ratiotype} is that 
 $$\eqref{eq:ratio_chi},\eqref{extendedElphickWocjan}, \eqref{eq:qtmratiotype} \leq \chi_{kq} \leq \chi_k.$$ 
 
 This implies that we can use the existing optimization methods for the bounds on the left which appeared in \cite{acfns20} (for the two inertial-type bounds) and in \cite{ANR2025} (for the Hoffman ratio-type bound). See Appendix for more details. For instance, for $k>1$ and for \eqref{extendedElphickWocjan}, we can use the MILP (27) from  \cite{acfns20} to find $\chi_{kq}$ of graphs for which $\chi_k=\chi_{kq}=\eqref{extendedElphickWocjan}$. Actually, this optimization method gives the exact value of $\chi_{kq}$ for several families of graphs, such as for the Kneser graphs $\chi_2(K(p,2))=n(n-1)/2$. See \cite[Section 4.2.1]{acfns20} for further details on other graph classes for which now we can obtain the value of the quantum distance-$k$ chromatic number. We should note that that the mentioned optimization methods for the three eigenvalue bounds only plays a role when $k>1$, since for $k=1$ it corresponds with $p(x)=x$. %Indeed, when $k=1$, the trace condition implies that $\sum_i p(\lambda_i) = a_1(\sum_i \lambda_i)+na_0 = na_0 = 0$, so we are simply counting the positive and negative eigenvalues. Also note that for the case $k=1$ and weighted adjacency matrices, the MILP (27) from \cite{acfns20} does not optimise either, since the trace is still zero. \\

Furthermore, several authors \cite{MR2016,L2024,GS2024} have been looking at separation type results for $\chi(G)$ and $\chi_q(G)$. In this regard, our eigenvalue bounds and their corresponding optimization methods can be used to obtain graphs that are not candidates to hold such separation.

 We end up this paper with several intriguing directions and questions. 
 \begin{itemize}
     % \item  We saw in Corollary \ref{coro:firstinertialinifinitedimprojectors} that the First inertial-type bound also hold in the more general situation when considering infinite dimension projectors (see \cite[Section 5]{GGL} for details). It remains open to show that this is also the case for the Second inertial-type bound and the Hoffman ratio-type bound.
     \item If for $k=1$ the quantum and the classical distance chromatic parameters coincide, does it imply anything for larger $k$? 
     \item Several authors have looked at separation results between $\chi_q$ and $\chi$ (see e.g. \cite{GS2024,L2024,MR2016,man13}). Using the fact that $\chi_{kq}(G) = \chi_{q}(G^{k})$, and that a $k$-quantum coloring of $G$ is a quantum coloring of $G^{k}$ (and vice-versa), it is immediate that \cite[Theorem 12]{man13}
     holds for $\chi_{kq}(G)$ as well. However, it is unclear whether a separation in the quantum $k$-distance chromatic number implies a separation in the quantum $\ell$-distance chromatic number for all $\ell \leq k$. 
     
     %Thus a natural question in this more general distance-$k$ setting is: can we also obtain analogous results if $k>1$? That is, the goal is to find constructions of infinite families of graphs for which $\chi_k > \chi_{kq}$. \aida{I assume that using power graphs on these known results does not immediately give anything right? (just double checking this question is not trivial in some instances)}
 \end{itemize}

%%%%%%%%%%%%%%%%%%%%%%%%%%%%%%%%%%%%%%%%%%%%%%%%%%%%%%%%%%%
\subsection*{Acknowledgments}
%%%%%%%%%%%%%%%%%%%%%%%%%%%%%%%%%%%%%%%%%%%%%%%%%%%%%%%%%%%
Aida Abiad is supported by the Dutch Research Council (NWO) through the grants VI.Vidi.213.085 and \linebreak OCENW.KLEIN.475. 
Benjamin Jany is supported by the Eindhoven Hendrik Casimir Institute (EHCI) of the Eindhoven University of Technology.

%%%%%%%%%%%%%%%%%%%%%%%%%%%%%%%%%%%%%%%%%%%%%%%%%%%%%%%%%%%

\newpage

%%%%%%%%%%%%%%%%%%%%%%%%%%%%%%%%%%%%%%%%%%%%%%%%%%%%%%%%%%%%%%%%%%%%%%%%%%%%%%%%%%%%%%%%%%%%%%%%%%%%%%%%%%%%%%%%%%%%%%
\section{Appendix}
%%%%%%%%%%%%%%%%%%%%%%%%%%%%%%%%%%%%%%%%%%%%%%%%%%%%%%%%%%%%%%%%%%%%%%%%%%%%%%%%%%%%%%%%%%%%%%%%%%%%%%%%%%%%%%%%%%%%%%
Since we showed that three eigenvalue bounds on the distance-$k$ chromatic number also hold in the quantum setting, one can use the existing optimization methods for such eigenvalue bounds in the classical setting. 

Indeed, the optimization of the First inertial-type bound (Theorem \ref{thmquantum:1st_inertial_chi}) and of the Second inertial-type bound (Theorem \ref{thmquantum:2nd_inertial_chi}) appeared in \cite{acfns20}. The optimization of the Hoffman ratio-type bound from Theorem \ref{thm:ratiotype} appeared in \cite{ANR2025}. Nevertheless, we decided to add them here for completeness.

%%%%%%%%%%%%%%%%%%%%%%%%%%%%%%%%%%%%%%%%%%%%%%%%%%%%%%%%%%%
\subsection{Optimization of the First inertial-type bound (Theorem \ref{thmquantum:1st_inertial_chi})}\label{subsec:opt-firstinertialtypebound}
%%%%%%%%%%%%%%%%%%%%%%%%%%%%%%%%%%%%%%%%%%%%%%%%%%%%%%%%%%%

Here we introduce a mixed-integer linear programming (MILP) formulation to compute the best polynomial for Theorem \ref{thmquantum:1st_inertial_chi}. Although solving a MILP is known to be NP-hard in general, we find that in practice our method effectively minimizes the bound in Theorem~\ref{thmquantum:1st_inertial_chi} for numerous interesting graphs. 
%We will see many examples later in Table~\ref{tab:ind1:comparisona2}.
%Moreover, the bound in Theorem~\ref{thmquantum:1st_inertial_chi} is also an upper bound for the quantum~$k$-independence number, which is not computable in general~\cite{mancinska162}. This further motivates our use of mixed-integer linear programming to minimize the bound in \ref{thmquantum:1st_inertial_chi}, as it provides a method to approximate this parameter.

Let~$\G$ have spectrum~$\sp \G=\big\{\theta_0^{[m_0]},\ldots, \theta_d^{[m_d]}\big\}$. Theorem~\ref{thmquantum:1st_inertial_chi} can be written in terms of these distinct eigenvalues and multiplicities as
\begin{equation} \label{eq:ind1:inertia1alt} \alpha_k(\G)\le \min\left\{\sum_{j: p(\theta_j)\ge w(p)} m_j , \sum_{j: p(\theta_j)\le W(p)} m_j\right\}.
\end{equation}
Equation~\eqref{eq:ind1:inertia1alt} only requires the computation of $p(\theta_j)$ for~$j=0,1,\dots,d$, instead of $p(\lambda_j)$ for all~$j\in[n]$. We will therefore base our MILP on this alternative formulation to reduce the number of variables and constraints. Note that Equation~\eqref{eq:ind1:inertia1alt} is invariant under scaling and translation of polynomial~$p$. Upon changing the sign of~$p$, we can therefore always assume we are minimizing~$\sum_{j: p(\theta_j)\ge w(p)} m_j$. Moreover, a constant can be added to~$p(x)$ such that~$w(p) = 0$.

Let~$p(x) = a_k x^k +\cdots + a_0$,~$\vecb=(b_0,\ldots,b_d) \in \{0,1\}^{d+1}$~and~$\vecm=(m_0,\ldots,m_d)$. As $w(p) = 0$, there exist a vertex~$u\in V(\G)$ such that~$p(A)_{uu} = 0$. Moreover, every other vertex~$v$ must satisfy~$p(A)_{vv}\ge 0$. The following MILP, with variables~$a_0,\ldots,a_k$ and~$b_1,\ldots, b_d$, formulates the problem of finding the best polynomial for the bound in Equation~\eqref{eq:ind1:inertia1alt} under the assumption that~$w(p) = p(A)_{uu} = 0$. To obtain the best upper bound on~$\alpha_k$, we iterate over all vertices~$u \in V(\G)$, solve the corresponding MILP and find the lowest objective value of all.

\begin{equation}
\boxed{\def\arraystretch{1.3}
\begin{array}{rll}
{\tt minimize} & \vecm^{\top} \vecb &\\
{\tt subject\ to} & \sum_{i = 0}^k a_i \cdot(A^i)_{vv} \geq 0, &v \in V(\G)\setminus \{u\}\\
 & \sum_{i = 0}^k a_i\cdot (A^i)_{uu} = 0, &\\
 & \sum_{i = 0}^k a_i \theta_j^{\ i} - M b_j + \varepsilon \leq 0, &j = 0,\dots,d\quad (\ast)\\
 & \vecb \in \{0,1\}^{d+1}&
		\end{array}}
\label{MILP:ind1:inertia}
\end{equation}

The constant~$M$ in MILP formulation~\eqref{MILP:ind1:inertia} is a large number and~$\varepsilon > 0$ small. The value of each variable~$b_j$ represents whether~$p(\theta_j) \ge w(p) = 0$. Constraint~$(\ast)$ ensures that $b_j=1$ if~$p(\theta_j)\ge 0$, and since~$\vecm^\top \vecb$ is minimized in the objective function, $b_j=1$ only if~$p(\theta_j)\ge 0$. So, upon minimizing the weighted sum of~$b_j$'s, we are optimizing the corresponding bound~$\alpha_k\le \vecm^{\top} \vecb$. We will see a concrete example of this MILP formulation for a specific graph later on in this section, in Example~\ref{ex:ind1:inertia1}.
 
As mentioned earlier, Equation~\eqref{eq:ind1:inertia1alt} is invariant under the scaling of~$p$. This means that we can always set~$\varepsilon = 1$ without loss of generality. If the chosen~$M$ is not large enough, the MILP will be infeasible and we can repeat with a larger~$M$.

If~$\G$ is a~$k$-partially walk-regular graph, all powers~$A^i$ of the adjacency matrix have constant diagonal. This means that~$w(p)=0$ if and only if~$p(A)_{uu}=0$ for all~$u\in V(\G)$, and hence
\[\tr p(A) = \sum_{u\in V(\G)} p(A)_{uu} = \sum_{j = 0}^d m_j p(\theta_j)= 0.\]
For $k$-partially walk-regular graphs, MILP formulation~\eqref{MILP:ind1:inertia} can therefore be simplified by replacing the first two constraints by $\sum_{j = 0}^d m_j p(\theta_j)= 0$, which results in MILP~\eqref{MILP:ind1:inertia}. As the constraints of this new formulation no longer depend on a chosen vertex~$u$, it suffices to solve a single MILP instance, whereas for general graphs we needed to solve one for every vertex.

	\begin{equation}
        \boxed{\def\arraystretch{1.3}
		\begin{array}{rl}
		{\tt minimize} & \vecm^{\top} \vecb\\
        {\tt subject\ to} & \sum_{j = 0}^d m_j \sum_{i = 0}^k a_i \theta_j^{\ i}= 0\\
		& \sum_{i = 0}^k a_i \theta_j^{\ i} - Mb_j + \varepsilon \leq 0, \quad j = 0,\dots,d \\
		& \vecb \in \{0,1\}^{d+1}
		\end{array}}
        \label{MILP:ind1:inertiaWR}
    \end{equation}

\begin{example}
    Let~$G=C_6$ with~$\sp G = \big\{2^{[1]},1^{[2]},-1^{[2]},-2^{[1]}\big\}$. This graph is 2-partially walk-regular (i.e., regular), hence we can apply MILP~\eqref{MILP:ind1:inertiaWR} to compute an upper bound on~$\alpha_2(\G)$. The objective function of MILP~\eqref{MILP:ind1:inertiaWR} is~$b_0+2b_1+2b_2+b_3$, which we want to minimize under the constraints
    \begin{alignat*}{7}
      6&a_0 & + &  0 && a_1 &+& &12& a_2 && = 0 && \\
	     &a_0 & + & 2 &&a_1 &+&  &4& a_2 && -Mb_0 &&+\varepsilon\le 0\\
      &a_0 & + & &&a_1 &+& &&a_2 && -Mb_1 &&+\varepsilon\le 0\\
      &a_0 & - & &&a_1 &+& && a_2 && -Mb_2 &&+\varepsilon\le 0\\
      &a_0 & - & 2 &&a_1 &+&  &4& a_2 && -Mb_3 &&+\varepsilon\le 0.
	\end{alignat*}
    The first constraint simplifies to~$a_0+2a_2=0$. By substituting this into the other constraints, we find
    \begin{alignat*}{9}
	    2 &a_1 & + & 2 &&a_2 -Mb_0 &&+\varepsilon\le 0 \hspace{1.5cm} &-a_1 & - & &&a_2 -Mb_2 &&+\varepsilon\le 0&&\\
        &a_1 & - & &&a_2 -Mb_1 &&+\varepsilon\le 0 \hspace{1.5cm} & -2a_1 & + &2 &&a_2 -Mb_3 &&+\varepsilon\le 0&&,
	\end{alignat*}
  which can be rewritten as
  \[\frac{1}{2}(Mb_0 - \varepsilon) \ge a_2+a_1 \ge -Mb_2+\varepsilon,\hspace{0.5cm} Mb_1 - \varepsilon \ge -a_2+a_1 \ge -\frac{1}{2}(Mb_3+\varepsilon). \]
  These two sandwiching inequalities imply that~$b_0$ and~$b_2$ cannot be zero simultaneously, and neither can~$b_1$ and~$b_3$. Since~$m_0$ and~$m_3$ have lowest multiplicity, the best objective value is obtained by setting~$b_0=b_3=1$,~$b_1=b_2=0$. The resulting upper bound is for~$\alpha_2(\G)$ equals two, which is tight, since~$\G$ admits a 2-independent set of size two. Vector~$\veca=(0,\frac{1}{2}(M-\varepsilon),0)$ complies with this choice of~$\vecb$ and satisfies the constraints, hence~$p_2(x) = \frac{1}{2}(M-\varepsilon)x$ is a corresponding optimal polynomial. 
  \label{ex:ind1:inertia1}
\end{example}

%%%%%%%%%%%%%%%%%%%%%%%%%%%%%%%%%%%%%%%%%%%%%%%%%%%%%%%%%%%
\subsection{Optimization of the Second inertial-type bound (Theorem \ref{thmquantum:2nd_inertial_chi})}\label{subsec:optsecondinertialbound}
%%%%%%%%%%%%%%%%%%%%%%%%%%%%%%%%%%%%%%%%%%%%%%%%%%%%%%%%%%%
  
We can use MILPs to optimize the polynomial~$p$ in Theorem~\ref{thmquantum:2nd_inertial_chi}. However, in this case we must solve~$n-1$ MILPs to obtain the best possible bound, whereas the first inertia-type bound only required one in case of $k$-partial walk-regularity. Let~$G$ have spectrum~$\sp G=\big\{\theta_0^{[m_0]},\ldots, \theta_d^{[m_d]}\big\}$ and let~$\vecm = (m_0,\dots,m_d)\in \{0,1\}^{d+1}$. For each~$\ell \in \{1,\dots,n-1\}$, we solve the following MILP. Note, however, that it may be infeasible for certain values of~$\ell$ if there is no subset of multiplicities adding up to~$\ell$.

\begin{equation}
\boxed{\def\arraystretch{1.3}
\begin{array}{rll}
{\tt maximize} & 1 + \frac{n-\vecm^{\top} \vecb}{\ell} \\
{\tt subject\ to} & \sum_{j=0}^d \sum_{i = 0}^k a_i m_j\theta_j^i = 0& \\
 & \sum_{i = 0}^k a_i \theta_j^i - M b_j+\varepsilon \leq 0,&j=0,\dots,d\\
 & \sum_{i = 0}^k a_i \theta_j^i - M c_j \leq 0, &j=0,\dots,d\\
  & \sum_{i = 0}^k a_i \theta_j^i + M (1-c_j) -\epsilon \geq 0,&j=0,\dots,d\\
 & \vecm^\top \vecc = \ell & \\
 & \vecb \in \{0,1\}^{d+1}, \quad \vecc \in \{0,1\}^{n}&
\end{array}
}
\label{MILP:ind1:inertia2}
\end{equation}

As before, the variables~$a_i$ are the coefficients of the polynomial of degree at most~$k$,~$p(x) = a_k x^k +\cdots + a_0$, and the first constraint is the hypothesis of Theorem~\ref{thmquantum:2nd_inertial_chi},~$\tr p(A) = 0$. The second set of constraints implies that~$b_j = 1$ if~$p(\lambda_j) \ge 0$. Moreover, as the objective function minimizes~$\vecm^\top \vecb$, we do not have~$b_j=1$ unless it is forced by the constraints. Therefore,~$p(\lambda_j) \ge 0$ if and only if~$b_j = 1$. Similarly, the third set of constraints implies that $c_j = 1$ if~$p(\lambda_j) > 0$. Since, contrary to~$\vecm^\top \vecb$, the value of~$\vecm^\top \vecc$ is not minimized by the MILP (in fact, we assume it to be constant), we need to explicitly add the fourth set of constraints to ensure that also~$p(\lambda_j) > 0$ whenever~$c_j = 1$. Note that this is a correction to MILP (27) in~\cite{acfns20}, where these constraints are missing.

Summarizing the above, we have
\begin{itemize}
\item
$|j:p(\lambda_j)> 0|=\vecm^{\top}\vecc = \ell$ (fifth constraint),
\item
$|j:p(\lambda_j)= 0|=\vecm^{\top}(\vecb-\vecc)$,
\item
$|j:p(\lambda_j)< 0|=n-\vecm^{\top} \vecb$.
\end{itemize}
This means that an optimal solution of MILP~\eqref{MILP:ind1:inertia2} indeed corresponds to the maximum value for the bound in Theorem~\ref{thmquantum:2nd_inertial_chi}.

\begin{example}
    Consider again the 2-partially walk-regular graph~$G=C_6$ with~$\sp G = \big\{2^{[1]},1^{[2]},-1^{[2]},-2^{[1]}\big\}$. Let~$\ell=3$. The first two sets of constraints of MILP~\eqref{MILP:ind1:inertia2} are the same as in MILP~\eqref{MILP:ind1:inertiaWR}, hence we know from Example~\ref{ex:ind1:inertia1} that they simplify to~$2a_2+a_0=0$,~$b_0+b_2\ge 1$ and~$b_1+b_3\ge 1$. The third, fourth and fifth set of constraints are
    
    \begin{minipage}[t]{0.45\textwidth}
        \begin{alignat*}{7}
	     &a_0 & + & 2 &&a_1 &+&  &4& a_2 && -Mc_0 &&\le 0\\
      &a_0 & + & &&a_1 &+& &&a_2 && -Mc_1 &&\le 0\\
      &a_0 & - & &&a_1 &+& && a_2 && -Mc_2 &&\le 0\\
      &a_0 & - &  2&&a_1 &+&  &4& a_2 && -Mc_3 &&\le 0
	\end{alignat*}
    \end{minipage}
    \begin{minipage}[t]{0.45\textwidth}
        \begin{alignat*}{7}
	     &a_0 & + & 2 &&a_1 &+&  &4& a_2 && +M(1-c_0) &&-\varepsilon\le 0\\
      &a_0 & + & &&a_1 &+& &&a_2 && +M(1-c_1) &&-\varepsilon\le 0\\
      &a_0 & - & &&a_1 &+& && a_2 && +M(1-c_2) &&-\varepsilon\le 0\\
      &a_0 & - & 2 &&a_1 &+&  &4& a_2 && +M(1-c_3) &&-\varepsilon\le 0\\
      &c_0 & + & 2 &&c_1 &+&  &2& c_2 & + & c_3 = \ell.&&
	\end{alignat*}
    \vspace{-0.4cm}
    \end{minipage}
     
 \noindent The fourth set of constraints combines to
  \[M(1-c_2) - \varepsilon \ge a_2+a_1 \ge -\frac{1}{2}M(1-c_0)+\varepsilon, \]
  and
   \[\frac{1}{2}M(1-c_3) - \varepsilon \ge -a_2+a_1 \ge -M(1-c_1)+\varepsilon, \]
  which implies that~$c_0+c_2\le 1$ and~$c_1+c_3\le 1$. The vectors~$\vecb=(1,1,0,0)$ and~$\vecc=(1,1,0,0)$ satisfy these conditions, as well as the other constraints. The corresponding objective value is~$1+\frac{6-3}{3} = 2$, which is not tight, as~$\chi_2(G)=3$. However, if we solve the MILP for all possible values of~$\ell$ using Gurobi, we find that this is the best possible value for the bound in Theorem~\ref{thmquantum:2nd_inertial_chi}.
\end{example}

The algorithm gives a lower bound for the actual maximum of \linebreak MILP~\eqref{MILP:ind1:inertia2}, as we restricted the optimal polynomial. Nevertheless, there are several graphs for which it gives a tight bound, and hence for which the bound in Theorem~\ref{thmquantum:2nd_inertial_chi} is tight, such as the Heawood graph, Klein 7-regular graph and Moebius-Kantor graph.

%More computational results can be found in Table~\ref{tab:ind1:chicompare} in Section~\ref{sec:ind1:ratio1}, where we compare the algorithm to a ratio-type bound.\aida{refer to Sjanne's thesis here} 

Like MILP~\eqref{MILP:ind1:inertia}, MILP~\eqref{MILP:ind1:inertia2} is tight for the incidence graphs of projective planes~$PG(2,q)$ with~$q$ a prime power and the prism graphs~$G_n$ with~$n\neq 2\mod 4$. Note that the latter are generalized Petersen graphs with parameters~$(n,1)$. The bound is also tight for (generalized) Petersen graphs with~$(n,k)\in \{(5,2),(8,3),(10,2)\}$. The second graph is also known as the M\"obius-Kantor graph and is walk-regular, but not distance-regular, hence Delsarte's LP bound~\cite{delsarte1973algebraic} is not applicable in this case.

 %%%%%%%%%%%%%%%%%%%%%%%%%%%%%%%%%%%%%%%
\subsection{Optimization of the Hoffman ratio-type bound (Theorem \ref{thm:ratiotype})}\label{LP:ratiotype}
%%%%%%%%%%%%%%%%%%%%%%%%%%%%%%%%%%%%%%%%%%

Let $G = (V,E)$ have adjacency matrix $A$ and distinct eigenvalues $\theta_0 > \cdots > \theta_d$. Note that we can scale by a positive number and translate the polynomial used in Theorem \ref{thm:ratiotype} without changing the value of the bound. Hence, we can assume $W(p) -\lambda(p) =1$. Furthermore, ~$\lambda(p) < W(p)$, so the scaling does not flip the sign of the bound. Hence, the problem reduces to finding the $p$ which maximizes $p(\lambda_1) - \lambda(p)$, subject to the constraint $W(p) -\lambda(p) =1$. For each~$u\in V$ and~$\ell \in [1, d]$, assume that~$W(p) = (p(A))_{uu}$,~$0=\lambda(p) = p(\theta_\ell)$ and solve the Linear Program (LP) below. The maximum of these~$dn$ solutions then equals the best possible bound obtained by Theorem \ref{thm:ratiotype}.

   {\scriptsize{ 
    \begin{equation}\label{eq:ratio_chi_LP}
        \boxed{\begin{aligned}
            \text{variables: } &(a_0, \hdots,a_k) \\
            \text{input: }&\text{The adjacency matrix } A \text{ and its distinct eigenvalues } \{\theta_0,\hdots,\theta_d\}. \\
            &\text{A vertex } u \in V, \text{ an } \ell \in [1,d]. \text{ An integer } k.\\
            \text{output: }&(a_0, \hdots,a_k) \text{, the coefficients of a polynomial } p \\ 
            \ \\
            \text{maximize} \ \ &\sum_{i=0}^k a_i \theta_0^i -\sum_{i=0}^k a_i\theta_\ell^i\\
            \text{subject to} \ \ & \sum_{i=0}^k a_i((A^i)_{vv} - (A^i)_{uu}) \leq 0 , \ v \in V \setminus \{u\} \\
            &\sum_{i=0}^k a_i((A^i)_{uu} - \theta_\ell^i) = 1\\
            &\sum_{i=0}^k a_i(\theta_0^i - \theta_j^i) > 0, \ \ j \in [1,d] \\
            &\sum_{i=0}^k a_i(\theta_j^i - \theta_\ell^i) \geq 0, \ \ j \in [1,d]
        \end{aligned}}
    \end{equation}
    }}
    
    Here the objective function is simply $p(\lambda_1) - \lambda(p)$. The first constraint says $(p(A))_{uu} \geq (p(A))_{vv}$ for all vertices $v \neq u$, which ensures $W(p) = (p(A))_{uu}$. The second constraint gives $p$ the correct scaling and translation such that $W(p) - \lambda(p) = 1$. The third constraint says $p(\theta_0) > p(\theta_j)$ for all $j \in [1,d]$, which ensures $p(\lambda_1) > \lambda(p)$. And the final constraint says $p(\theta_\ell) \leq p(\theta_j)$ for all $j \in [1,d]$, which ensures $\lambda(p) = p(\theta_\ell)$.

\end{document}